\documentclass[11pt]{amsart}
\usepackage{amscd}
\usepackage{amsmath,empheq}
\usepackage{amsfonts}
\usepackage{amssymb}
\usepackage{mathrsfs}

\usepackage[all]{xy}

\usepackage{mathtools}

\newtheorem{theorem}{Theorem}

\newtheorem{lemma}[theorem]{Lemma}
\newtheorem{prop}[theorem]{Proposition}
\newtheorem{remark}{Remark}

\newenvironment{proof-sketch}{\noindent{\bf Sketch of Proof}\hspace*{1em}}{\qed\bigskip}

\everymath{\displaystyle}

\newcommand{\RR}{\mathbb R}

\newcommand{\NN}{\mathbb N}

\newcommand{\di}{\displaystyle}

\newcommand{\divv}{{\rm div}}

\newcommand{\bb}{\begin{equation}}
\newcommand{\bbb}{\end{equation}}

\renewcommand{\leq}{\leqslant}

\renewcommand{\geq}{\geqslant}

\begin{document}

\title{Anisotropic $(p,q)$-equations with \\ gradient dependent reaction}

\author[N.S. Papageorgiou]{N.S. Papageorgiou}
\address[N.S. Papageorgiou]{National Technical University, Department of Mathematics,
				Zografou Campus, Athens 15780, Greece
				\& Institute of Mathematics, Physics and Mechanics, 1000 Ljubljana, Slovenia
				}
\email{\tt npapg@math.ntua.gr}

\author[V.D. R\u{a}dulescu]{V.D. R\u{a}dulescu}
\address[V.D. R\u{a}dulescu]{Faculty of Applied Mathematics, AGH University of Science and Technology, al. Mickiewicza 30, 30-059 Krak\'ow, Poland \& Department of Mathematics, University of Craiova, 200585 Craiova, Romania
\& Institute of Mathematics, Physics and Mechanics, 1000 Ljubljana, Slovenia
}
\email{\tt radulescu@inf.ucv.ro}

\author[D.D. Repov\v{s}]{D.D. Repov\v{s}}
\address[D.D. Repov\v{s}]{Faculty of Education and Faculty of Mathematics and Physics, University of Ljubljana, 
\& Institute of Mathematics, Physics and Mechanics, 1000 Ljubljana, Slovenia}
\email{\tt dusan.repovs@guest.arnes.si}


\begin{abstract}
We consider a Dirichlet problem driven by the anisotropic $(p,q)$-Laplacian and a reaction with gradient dependence (convection).
The presence of the gradient in the source term excludes from consideration a variational approach in dealing with the qualitative analysis of this problem with unbalanced growth.
Using the frozen variable method and eventually a fixed point theorem, the main result of this paper establishes that the problem has a positive smooth solution.
\end{abstract}

\maketitle

\section{Introduction and statement of the problem}
In this paper, we are concerned with the study of a nonlinear anisotropic problem whose features are the following:\\ (i) the presence of several differential operators with different growth, which generates a {\it double phase} associated energy;\\ (ii) the reaction combines the multiple effects produced by a convection (gradient) term and a nonlinearity with nonuniform nonresonance near the principal eigenvalue;\\ (iii) the problem has both a nonvariational structure (generated by the convection nonlinearity) and an anisotropic framework (created by the presence of two differential operators with variable exponent and a gradient term with variable potential);\\ (iv) due to the particular structure of the problem studied in this paper, we develop an approach based on the frozen variable method.

 Roughly speaking, the main result of this paper establishes
the following properties:\\ (a)  the minimal solution map is compact;\\ (b) the problem has a positive smooth solution, which is a fixed point of the minimal solution map.

Summarizing, this paper
 is concerned with the refined qualitative  analysis of solutions for a class of {\it nonvariation} problems driven by anisotropic differential operators with {\it unbalanced anisotropic growth}.

 We recall in what follows some of the outstanding contributions of the Italian school  to the study of unbalanced integral functionals and double phase problems.  We first refer to the pioneering contributions of Marcellini \cite{marce1,marce2,marce3} who studied lower semicontinuity and regularity properties of minimizers of certain quasiconvex integrals. Problems of this type arise in nonlinear elasticity and are connected with the deformation of an elastic body, cf. Ball \cite{ball1,ball2}. We also refer to Fusco and Sbordone \cite{fusco} for the study of regularity of minima of anisotropic integrals.

In order to recall the roots of double phase problems, let us assume that $\Omega$ is a bounded domain in ${\mathbb R}^N$ ($N\geq 2$) with smooth boundary. If $u:\Omega\to{\mathbb R}^N$ is the displacement and if $Du$ is the $N\times N$  matrix of the deformation gradient, then the total energy can be represented by an integral of the type
\begin{equation}\label{paolo}I(u)=\int_{\Omega} f(x,Du(x))dx,\end{equation}
where the energy function $f=f(x,\xi):\Omega\times{\mathbb R}^{N\times N}\to{\mathbb R}$ is quasiconvex with respect to $\xi$. One of the simplest examples considered by Ball is given by functions $f$ of the type
$$f(\xi)=g(\xi)+h({\rm det}\,\xi),$$
where ${\rm det}\,\xi$ is the determinant of the $N\times N$ matrix $\xi$, and $g$, $h$ are nonnegative convex functions, which satisfy the growth conditions
$$g(\xi)\geq c_1\,|\xi|^q;\quad\lim_{t\to+\infty}h(t)=+\infty,$$
where $c_1$ is a positive constant and $1<q<N$. The condition $q< N$ is necessary to study the existence of equilibrium solutions with cavities, that is, minima of the integral \eqref{paolo} that are discontinuous at one point where a cavity forms; in fact, every $u$ with finite energy belongs to the Sobolev space $W^{1,q}(\Omega,{\mathbb R}^N)$, and thus it is a continuous function if $q>N$. In accordance with these problems arising in nonlinear elasticity, Marcellini \cite{marce1,marce2} considered continuous functions $f=f(x,u)$ with {\it unbalanced growth} that satisfy
$$c_1\,|u|^q\leq |f(x,u)|\leq c_2\,(1+|u|^p)\quad\mbox{for all}\ (x,u)\in\Omega\times{\mathbb R},$$
where $c_1$, $c_2$ are positive constants and $1\leq q\leq p$. Regularity and existence of solutions of elliptic equations with $p,q$--growth conditions were studied in \cite{marce2}.

The study of non-autonomous functionals characterized by the fact that the energy density changes its ellipticity and growth properties according to the point has been continued in a series of remarkable papers by Mingione {\it et al.} \cite{baroni0}--\cite{baroni}, \cite{colombo0}--\cite{colombo1}. We also refer to Mingione and R\u adulescu \cite{minrad} for
an overview of recent results concerning elliptic variational problems with nonstandard growth conditions and related to different kinds of nonuniformly elliptic operators. 
These contributions are in relationship with the works of Zhikov \cite{zhikov1}, in order to describe the
behavior of phenomena arising in nonlinear
elasticity.
In fact, Zhikov intended to provide models for strongly anisotropic materials in the context of homogenisation.
 In particular, Zhikov considered the following model of
functional  in relationship to the Lavrentiev phenomenon:
$$
{\mathcal P}_{p,q}(u) :=\int_{\Omega} (|\nabla u|^q+a(x)|\nabla u|^p)dx,\quad 0\leq a(x)\leq L,\ 1< q< p.
$$
In this functional, the modulating coefficient $a(x)$ dictates the geometry of the composite made by
two differential materials, with hardening exponents $p$ and $q$, respectively.

The functional ${\mathcal P}_{p,q}$ falls in the realm of the so-called functionals with
nonstandard growth conditions of $(p, q)$--type, according to Marcellini's terminology. This is a functional of the type in \eqref{paolo}, where the energy density satisfies
$$|\xi|^q\leq f(x,\xi)\leq  |\xi|^p+1,\quad 1\leq q\leq p.$$

Another significant model example of a functional with $(p,q)$--growth studied by Mingione {\it et al.} is given by
$$u\mapsto \int_{\Omega} |\nabla u|^q\log (1+|\nabla u|)dx,\quad q\geq 1,$$
which is a logarithmic perturbation of the $p$-Dirichlet energy.

The purpose of this paper is to study the following anisotropic Dirichlet problem with a reaction which depends on the gradient (convection):
\begin{equation}\label{eq1}
  \left\{
\begin{array}{lll}
-\Delta_{p(z)}u(z)-\Delta_{q(z)}u(z)=\hat{r}(z)|Du(z)|^{\tau(z)-1}+f(z,u(z))\text{ in } \Omega,\\
u|_{\partial\Omega}=0,\ u\geq0,\ q(\cdot)<p(\cdot),
\end{array}
\right.
\end{equation}
 where $\Omega\subseteq\RR^N$ is a bounded domain with a $C^2$-boundary $\partial\Omega$.

For $r\in C(\overline{\Omega})$ we define
$$
r_-=\min_{\overline{\Omega}}r \mbox{ and }r_+=\max_{\overline{\Omega}}r.
$$

We consider the set $E_1=\{\hat{r}\in C(\overline{\Omega}):\:1<r_-\}$. For $r\in E_1$, we denote by $\Delta_{r(z)}$   the $r$-Laplace differential operator defined by
$$
\Delta_{r(z)}u=\divv (|Du|^{r(z)-2}Du)\mbox{ for all } u\in W^{1,r(z)}_0(\Omega).
$$

In problem \eqref{eq1}, the differential operator is the sum of two such operators (anisotropic $(p,q)$-equation or anisotropic double phase problem). The reaction (right-hand side of \eqref{eq1}) depends also on the gradient of $u$ (convection). This makes the problem nonvariational, which means that eventually our proof should be topological, based on the fixed point theory. We assume that $q_+<p_-$.

Recently there have been some existence results for elliptic equations driven by the $(p,q)$-Laplacian (or even more general nonhomogeneous operators) and a gradient dependent reaction. We mention the works of Bai \cite{2Bai}, Bai, Gasi\'nski and Papageorgiou \cite{3Bai-Gas-Pap}, Candito, Gasi\'nski and Papageorgiou \cite{5Can-Gas-Pap}, Faria, Miyagaki and Motreanu \cite{11Far-Miy-Mot}, Gasi\'nski, Krech and Papageorgiou \cite{13Gas-Kre-Pap}, Gasi\'nski and Winkert \cite{16Gas-Win}, Liu and Papageorgiou \cite{18Liu-Pap}, Marano and Winkert \cite{19Mar-Win}, Papageorgiou, R\u adulescu and Repov\v s \cite{22Pap-Rad-Rep}, Papageorgiou, Vetro and Vetro \cite{24Pap-Vet-Vet}, and Zeng, Liu and Migorski \cite{30Zen-Liu-Mig}. All the aforementioned papers deal with isotropic equations. To the best of our knowledge, there are no works in the literature dealing with anisotropic $(p,q)$-equations with convection. It appears that our existence theorem here is the first such result.

We mention that equations driven by sum of two differential operators of different nature, appear in mathematical models of physical processes. We refer to the works of Bahrouni, R\u adulescu and Repov\v s \cite{1Bah-Rad-Rep}, Cencelj, R\u adulescu and Repov\v s \cite{6Cen-Rad-Rep}, R\u adulescu \cite{26Rad}, Zhikov \cite{32Zhi} and the references therein.

Our approach is based on the so-called ``frozen variable method". According to this method, in the reaction we fix (freeze) the gradient term and in this way we have a variational problem, which can be treated using tools from the critical point theory. We need to find a canonical way to choose a solution from the solution set of the ``frozen problem". To this end, we show that the ``frozen problem" has a smallest positive solution (minimal positive solution). In this way we can define the minimal solution map. Using an iterative process, we show that this map is compact and then using the Leray-Schauder Alternative Principle, we produce a fixed point for the minimal solution map. This fixed point is the desired positive solution of \eqref{eq1}.

The main features of this paper are the presence of the convection term $|Du|^{r(\cdot)-1}$ (inducing a nonvariational structure of the problem) and the combined effects
generated by the variable exponents $p(\cdot )$ and $q(\cdot )$ (producing an anisotropic abstract setting, which describes patterns associated with strongly anisotropic materials). We also highlight that the growth of the variable exponent $\tau(\cdot)$ associated with the convection term is consistent with the literature (we only restrict $\tau_+<p_-$, see hypotheses $H_0$). We recall that in the case of the usual Laplace operator, as remarked by Serrin \cite{serrin}, Choquet-Bruhat and Leray \cite{choquet}, and Kazdan and Warner \cite{kazdan}, a basic requirement is that the
convection term grows at most quadratically; this is a natural hypothesis in order to apply the maximum principle.

\section{Mathematical background and hypotheses}

The analysis of problem \eqref{eq1} uses Lebesgue and Sobolev spaces with variable exponents. A comprehensive treatment of such spaces can be found in the book of Diening, Harjulehto, H\"asto and Ruzicka \cite{7Die-Har-Has-Ruz}.

Let $M(\Omega)$ be the space of all measurable functions $u:\Omega\to\RR$. As usual, we identify two such functions which differ only on a Lebesgue null set. Let $r\in E_1$. Then the variable exponent Lebesgue space $L^{r(z)}(\Omega)$ is defined as follows
$$
L^{r(z)}(\Omega)=\left\{u\in M(\Omega):\:\int_\Omega |u|^{r(z)}dz<\infty\right\}.
$$

This space is equipped with the so called ``Luxemburg norm" defined by
$$
\|u\|_{r(z)}=\inf\left\{\lambda>0:\: \int_\Omega \left[\frac{|u|}{\lambda}\right]^{r(z)}dz\leq1\right\}.
$$

The space  $\left(L^{r(z)}(\Omega),\|\cdot\|_{r(z)}\right)$ is separable and uniformly convex (thus, reflexive by the Milman-Pettis theorem, see  Theorem 3.4.28 of Papageorgiou and Winkert \cite[p. 225]{25Pap-Win}). Let $r'\in E_1$ be defined by $r'(z)=\frac{r(z)}{r(z)-1}$ $z\in\overline{\Omega}$ (the conjugate variable exponent to $r(\cdot)$). We have $L^{r(z)}(\Omega)^*=L^{r'(z)}(\Omega)$ and also the following version of the H\"older inequality is true
$$
\int_\Omega |uv|dz\leq \left[\frac{1}{r_-}+\frac{1}{r'_-}\right]\|u\|_{r(z)}\|v\|_{r'(z)}
$$
for all $u\in L^{r(z)}(\Omega)$, $v\in L^{r'(z)}(\Omega)$.

Suppose that $r_1,r_2\in E_1$ and we have $r_1(z)\leq r_2(z)$ for all $z\in\overline{\Omega}$. Then
$$
L^{r_2(z)}(\Omega)\hookrightarrow L^{r_1(z)}(\Omega) \mbox{ continuously. }
$$

Having the variable exponent Lebesgue spaces, we can define in the usual way the corresponding variable exponent Sobolev spaces. So, given $r\in E_1$ the variable exponent Sobolev space $W^{1,r(z)}(\Omega)$ is defined by
$$
W^{1,r(z)}(\Omega)=\left\{u\in L^{r(z)}(\Omega):\:|Du|\in L^{r(z)}(\Omega)\right\}.
$$

Here, the gradient $Du$ is understood in the weak sense. The space $W^{1,r(z)}(\Omega)$ is furnished with the norm
$$
\|u\|_{1,r(z)}=\|u\|_{r(z)}+\| \, |Du|\, \|_{r(z)} \mbox{ for all }u\in W^{1,r(z)}(\Omega).
$$

For simplicity, in the sequel we write $\|Du\|_{r(z)}=\|\, |Du|\, \|_{r(z)}$.

Also, if $r\in E_1$ is Lipschitz continuous (that is, $r\in E_1\cap C^{0,1}(\overline{\Omega})$), then we define
$$
W^{1,r(z)}_0(\Omega)=\overline{C_c^\infty(\Omega)}^{\|\cdot\|_{1,r(z)}}.
$$

Both spaces $W^{1,r(z)}(\Omega)$ and $W^{1,r(z)}_0(\Omega)$ are separable and uniformly convex (thus, reflexive). For the space $W^{1,r(z)}_0(\Omega)$ the Poincar\' e inequality is valid, namely there exists $\hat{C}>0$ such that
$$
\|u\|_{r(z)}\leq \hat{C}\|Du\|_{r(z)} \mbox{ for all }u\in W^{1,r(z)}_0(\Omega).
$$

This means that on $W^{1,r(z)}_0(\Omega)$ we can consider the equivalent norm
$$
\|u\|_{1,r(z)}=\|Du\|_{r(z)} \mbox{ for all } u\in W^{1,r(z)}_0(\Omega).
$$

Given $r\in E_1$, we introduce the critical variable exponent $r^*(\cdot)$ corresponding to $r(\cdot)$, defined by
$$
r^*(z)=\left\{
         \begin{array}{ll}
           \frac{Nr(z)}{N-r(z)}, & \hbox{ if } r(z)<N \\
           +\infty, & \hbox{ if } N\leq r(z)
         \end{array}
       \right. \mbox{ for all }z\in\overline{\Omega}.
$$

Consider $r\in E_1\cap C^{0,1}(\overline{\Omega})$, $q\in E_1$ with $q_+<N$ and assume that $1<q(z)\leq r^*(z)$ (resp., $1<q(z)<r^*(z)$) for all $z\in\overline{\Omega}$. Then we have the following embeddings (anisotropic Sobolev embedding theorem):
$$
W^{1,r(z)}_0(\Omega)\hookrightarrow L^{q(z)}(\Omega) \mbox{ continuously }
$$
$$
\left(\mbox{ resp., } W^{1,r(z)}_0(\Omega)\hookrightarrow L^{q(z)}(\Omega) \mbox{ compactly}\right).
$$

The following modular function is very useful in the study of the variable exponent spaces
$$
\rho_r(u)=\int_\Omega |u|^{r(z)} dz \mbox{ for all }u\in L^{r(z)}(\Omega)\;\;(r\in E_1).
$$

Again we write $\rho_r(Du)=\rho_r(|Du|)$.

This function is closely related to the norm.

\begin{prop}\label{prop1}
  If $r\in E_1$ and $\{u, u_n\}_{n\in\NN}\subseteq L^{r(z)}(\Omega)$, then
\begin{itemize}
  \item[(a)] $\|u\|_{r(z)}=\mu\Leftrightarrow\rho_r\left(\frac{u}{\mu}\right)=1$;
  \item[(b)] $\|u\|_{r(z)}<1$  (resp. $=1$, $>1$)  $\Leftrightarrow$ $\rho_r(u)<1$ (resp. $=1$, $>1$);
  \item[(c)] $\|u\|_{r(z)}<1$ $\Rightarrow$ $\|u\|_{r(z)}^{r_+}\leq \rho_r(u)\leq\|u\|_{r(z)}^{r_-}$;
  \item[(d)] $\|u\|_{r(z)}>1$ $\Rightarrow$ $\|u\|_{r(z)}^{r_-}\leq \rho_r(u)\leq\|u\|_{r(z)}^{r_+}$;
  \item[(e)] $\|u_n\|_{r(z)}\to0$ $\Leftrightarrow$ $\rho_r(u_n)\to0$;
  \item[(f)] $\|u_n\|_{r(z)}\to+\infty$ $\Leftrightarrow$ $\rho_r(u_n)\to+\infty$.
\end{itemize}
\end{prop}

Given $r\in E_1\cap C^{0,1}(\overline{\Omega})$, we have
$$
W^{1,r(z)}_0(\Omega)^*=W^{-1,r'(z)}(\Omega).
$$

Then we introduce the nonlinear map $A_{r(z)}: W^{1,r(z)}_0(\Omega)\to W^{-1,r'(z)}(\Omega)= W^{1,r(z)}_0(\Omega)^*$ defined by
$$
\langle A_{r(z)}(u),h\rangle=\int_\Omega |Du|^{r(z)-2}(Du,Dh)_{\RR^N}dz
$$
for all $u,h \in W^{1,r(z)}_0(\Omega)$.

This operator has the following properties (see Gasi\'nski and Papageorgiou \cite[Proposition 2.5]{15Gas-Pap} and R\u adulescu-Repov\v s \cite[p. 40]{27Rad-Rep}).

\begin{prop}\label{prop2}
  The operator $A_{r(z)}:W^{1,r(z)}_0(\Omega)\to W^{-1,r'(z)}(\Omega)$ is bounded (that is, maps bounded sets to bounded sets), continuous, strictly monotone (hence maximal monotone too) and of type $(S)_+$, that is,
$$
``u_n\overset{w}{\to}u \mbox{ in }W^{1,r(z)}_0(\Omega) \mbox{ and }\limsup_{n\to\infty}\langle A_{r(z)}(u_n),u_n-u\rangle\leq0
$$
$$
\Downarrow
$$
$$
u_n\to u \mbox{ in } W^{1,r(z)}_0(\Omega)".
$$
\end{prop}

We will also use the Banach space $C_0^1(\overline{\Omega})=\{u\in C^1(\overline{\Omega}):\:u|_{\partial\Omega}=0\}$. This is an ordered Banach space with positive cone $C_+=\left\{u\in C_0^1(\overline{\Omega}):\: u(z)\geq0 \mbox{ for all }z\in\overline{\Omega}\right\}$. This cone has a nonempty interior given by
$$
{\rm int}\,C_+=\left\{u\in C_+:\:u(z)>0 \mbox{ for all }z\in\Omega,\ \frac{\partial u}{\partial n}|_{\partial \Omega}<0\right\}
$$
with $n(\cdot)$ being the outward unit normal on $\partial \Omega$.

Consider the following anisotropic eigenvalue problem
\begin{equation}\label{eq2}
  -\Delta_{p(z)} u(z)=\hat{\lambda}|u(z)|^{p(z)-2}u(z) \mbox{ in } \Omega,\;u|_{\partial\Omega}=0,
\end{equation}
with $p\in E_1$. We say that $(\hat{\lambda},\hat{u})\in \RR\times\left( W^{1,r(z)}_0(\Omega)\setminus\{0\} \right)$ is an ``eigenpair" for problem \eqref{eq2}, if
$$
\langle A_{p(z)}(\hat{u}),h\rangle=\hat{\lambda}\int_\Omega |\hat{u}(z)|^{p(z)-2}\hat{u}(z)h(z)dz \mbox{ for all } h\in W^{1,r(z)}_0(\Omega).
$$

Then $\hat{\lambda}$ is an ``eigenvalue" and $\hat{u}\not=0$ is a corresponding ``eigenfunction". We let
$$
\mathcal{L}=\{\hat{\lambda}\in\RR:\:\hat{\lambda} \mbox{ is an eigenvalue of \eqref{eq2}}\}.
$$

For the anisotropic eigenvalue problem, in contrast to the isotropic one, we can have $\inf\mathcal{L}=0$ (see Fan, Zhang and Zhao \cite[Theorem 3.1]{9Fan-Zha-Zha}). If we can find $\eta\in\RR^N$ ($N>1$) such that for all $z\in\Omega$, the function $t\mapsto \vartheta(t)=p(z+t\eta)$ is monotone on $T_z=\{t\in \RR:\:z+t\eta\in \Omega\}$ and $p\in C^1(\overline{\Omega})$, then problem \eqref{eq2} has a principal eigenvalue $\hat{\lambda}_1>0$ with corresponding positive eigenfunction $\hat{u}_1\in {\rm int}\,C_+$ (see Fan, Zhang and Zhao \cite[Theorem 3.3]{9Fan-Zha-Zha}) (see also Byun and Ko \cite{4Byu-Ko} and Fan \cite{8Fan}). We have
\begin{equation}\label{eq3}
  0<\hat{\lambda}_1=\frac{\rho_p(D\hat{u}_1)}{\rho_p(\hat{u}_1)}\leq
\frac{\rho_p(Du)}{\rho_p(u)} \mbox{ for all }\mu\in W^{1,r(z)}_0(\Omega),\ u\not=0.
\end{equation}

As we already mentioned in the Introduction, our approach is eventually topological and uses the ``Leray-Schauder Alternative Principle".

Given a Banach space $X$, a map $\xi:X\to X$ is said to be ``compact" if it is continuous and maps bounded sets into relatively compact sets. If $X$ is reflexive and $\xi:X\to X$ is completely continuous (that is, $x_n\overset{w}{\to}x$ in $X$ $\Rightarrow$ $\xi(x_n)\to \xi(x)$), then $\xi(\cdot)$ is compact (see Proposition 3.1.7 of Gasi\'nski and Papageorgiou \cite[p. 268]{14Gas-Pap}). The Leray-Schauder Alternative Principle asserts the following property.
\begin{theorem}\label{th3}
  If $X$ is a Banach space, $\xi:X\to X$ is compact and
$$
D(\xi)=\{u\in X:\:u=t\xi(u),\ 0<t<1\},
$$
then one of the following statements is true:
\begin{itemize}
  \item[(a)] $D(\xi)$ is unbounded;
  \item[(b)] $\xi(\cdot)$ has a fixed point.
\end{itemize}
\end{theorem}

Throughout this paper we will use the following notation. By $\|\cdot\|$ we denote the norm of the Sobolev space $W^{1,p(z)}_0(\Omega)$. If $p\in E_1\cap C^{0,1}(\overline{\Omega})$, then by the Poincar\'e inequality, we have
$$
\|u\|=\|Du\|_{p(z)} \mbox{ for all }u\in W^{1,p(z)}_0(\Omega).
$$

Given $u\in W^{1,p(z)}_0(\Omega)$, $u\geq0$,  we denote by $[0,u]$ the order interval
$$
[0,u]=\left\{h\in W^{1,p(z)}_0(\Omega):\:0\leq h(z)\leq u(z) \mbox{ for a.a. }z\in \Omega\right\}.
$$

Let $g:\Omega\times\RR\to \RR$ be a measurable function. By $N_g(\cdot)$ we denote the Nemytski (superposition) operator defined by
$$
N_g(u)(\cdot)=g(\cdot,u(\cdot)) \mbox{ for all } u:\Omega\to\RR \mbox{ measurable. }
$$

Evidently, $z\mapsto N_g(u)(z)$ is measurable. Recall that if $g:\Omega\times\RR\to\RR$ is a Carath\'eodory function (that is, for all $x\in \RR$, the mapping $z\mapsto g(z,x)$ is measurable and for a.a. $z\in\Omega$, the mapping $x\mapsto g(z,x)$ is continuous), then $(z,x)\mapsto g(z,x)$ is measurable (see Gasi\'nski and Papageorgiou \cite[p. 405]{14Gas-Pap}).

For every $x\in\RR$, we set $x^\pm=\max\{\pm x,0\}$ and then for $u\in W^{1,p(z)}_0(\Omega)$ we define $u^\pm(\cdot)=u(\cdot)^\pm$. We know that
$$
u^\pm\in W^{1,p(z)}_0(\Omega),\;u=u^+-u^-,\;|u|=u^+ +u^-.
$$

A set $S\subseteq W^{1,p(z)}_0(\Omega)$ is said to be ``downward directed", if for every pair $u_1,u_2\in S$, we can find $u\in S$ such that $u\leq u_1$, $u\leq u_2$.

\smallskip
Our hypotheses on the data of problem \eqref{eq1} are the following:

\smallskip
$H_0$: $p\in C^1(\overline{\Omega})$, there exists a vector $\hat{\eta}\in \RR^N$ such that for all $z\in\Omega$, the function $t\mapsto p(z+t\hat{\eta})$ is monotone on $I_z=\{t\in \RR:\: z+t\hat{\eta}\in \Omega\}$, $q\in E_1\cap C^{0,1}(\overline{\Omega})$, $\tau\in E_1$, $\tau_+<p_-\leq p_+<p_-^*$, $0\leq p_+-p_-\leq1$, $q_+<p_-$, and $\hat{r}\in L^\infty(\Omega)$, $\hat{r}(z)\geq0$ for a.a. $z\in \Omega$, $\hat{r}\not=0$.

\begin{remark}
  As we already mentioned earlier in this section, the hypotheses on the exponent $p(\cdot)$, imply that the eigenvalue problem \eqref{eq2} has a principal eigenvalue $\hat{\lambda}_1>0$ and an associated positive eigenfunction $\hat{u}_1\in{\rm int}\,C_+$ (see Fan, Zhang and Zhao \cite{9Fan-Zha-Zha}, Fan \cite{8Fan}, and Byun and Ko \cite{4Byu-Ko}).
\end{remark}

The hypotheses on the perturbation $f(z,x)$ are:

\smallskip
$H_1$: $f:\Omega\times\RR\to\RR$ is a Carath\'eodory function, $f(z,0)=0$ for a.a. $z\in\Omega$ and
\begin{itemize}
  \item[$(i)$] for every $\rho>0$, there exists $a_\rho\in L^\infty(\Omega)$ such that
$$
0\leq f(z,x)\leq a_\rho(z) \mbox{ for a.a. $z\in \Omega$, all $0\leq x\leq\rho$;}
$$
  \item[$(ii)$] there exists a function $\vartheta\in L^\infty(\Omega)$ such that
$$
\vartheta(z)\leq \hat{\lambda}_1 \mbox{ for a.a. } z\in\Omega, \vartheta\not\equiv\hat{\lambda}_1,
$$
$$
\limsup_{x\to+\infty}\frac{f(z,x)}{x^{p(z)-1}}\leq \vartheta(z) \mbox{ and } \limsup_{x\to+\infty}\frac{p_+ F(z,x)}{x^{p(z)}}\leq \vartheta(z)
$$
uniformly for a.a. $z\in\Omega$, with $F(z,x)=\displaystyle{\int_0^x f(z,s)ds}$;
  \item[$(iii)$] there exist $\eta_0>0$ and $M>0$ such that
$$
-\eta_0\leq \hat{\lambda}_1 x^{p(z)}-p_+F(z,x) \mbox{ for a.a. $z\in\Omega$, all $x\geq M$;}
$$
  \item[$(iv)$] there exist $\mu\in E_1$ with $\mu_+<q_-$ and $\delta>0$ such that
$$
C_0 x^{\mu(z)-1}\leq f(z,x) \mbox{ for a.a. }z\in\Omega, \mbox{ all }0\leq x\leq \delta, \mbox{ some }C_0>0;
$$
  \item[$(v)$] for a.a. $z\in \Omega$, all $x\in\RR$ and all $t\in (0,1)$ we have
$$
f\left(z,\frac{1}{t}x\right)\leq \frac{1}{t^{p(z)-1}}f(z,x).
$$
\end{itemize}

\begin{remark}
  Hypothesis $H_1(ii)$ implies that at $+\infty$ we have nonuniform nonresonance with respect to the principal eigenvalue $\hat{\lambda}_1>0$. Hypothesis $H_1(iv)$ implies the presence of a concave term near zero. Hypothesis $H_1(v)$ is satisfied if for example for a.a. $z\in\Omega$ the quotient function $x\mapsto \frac{f(z,x)}{x^{p(z)-1}}$ is nondecreasing on $\overset{\circ}{\RR}_+=(0,+\infty)$.
\end{remark}

The following lemma will help us to exploit the nonuniform nonresonance condition in hypothesis $H_1(ii)$.

\begin{lemma}\label{lem4}
  If $\vartheta\in L^\infty(\Omega)$, $\vartheta(z)\leq \hat{\lambda}_1$ for a.a. $z\in\Omega$ and $\vartheta\not\equiv\hat{\lambda}_1$, then there exists $C_1\in (0,1)$ such that
$$
C_1\rho_p(Du)\leq \rho_p(Du)-\int_\Omega \vartheta(z)|u|^{p(z)}dz
$$
for all $u\in W^{1,p(z)}_0(\Omega)$.
\end{lemma}

\begin{proof}
Consider the eigenvalue problem
$$\left\{\begin{array}{lll}
\di &-\Delta_{p(z)}u=\tilde\lambda \vartheta (z)|u|^{p(z)-2}u&\ \mbox{in}\ \Omega\\
&u=0\ &\mbox{on}\ \partial\Omega.\end{array}\right.
$$

From Fan, Zhang and Zhao \cite{9Fan-Zha-Zha}, we know that this anisotropic eigenvalue problem has a principal eigenvalue $\hat{\lambda}_1>1$ and we have
$$\rho_p(Du)\geq\hat{\lambda}_1\int_\Omega \vartheta (z)|u|^{p(z)}dz\ \mbox{for all}\ u\in W_0^{1,p(z)}(\Omega).$$
Then for all $u\in W_0^{1,p(z)}(\Omega)$, we have
$$\rho_p(Du)-\int_\Omega \vartheta (z)|u|^{p(z)}dz\geq\left(1-\frac{1}{\hat{\lambda}_1} \right)\rho_p(Du)=C_1\rho_p(Du),$$
with $C_1=\frac{\hat{\lambda}_1-1}{\hat{\lambda}_1}\in (0,1)$. This proves the lemma.
\end{proof}

\begin{remark}\label{caracal}
We have a similar inequality for the corresponding norms. Indeed, if
$$\|u\|_*=\inf\left\{\lambda>0;\ \int_\Omega \vartheta (z)\left| \frac{u(z)}{\lambda}\right|^{p(z)}dz\leq 1 \right\},$$
then there exists $C\in (0,1)$ such that $\|u\|_*\leq C\,\|u\|$ for all $u\in W_0^{1,p(z)}(\Omega)$. Arguing by contradiction, suppose we could find $\{u_n\}_{n\in\NN}\subseteq W_0^{1,p(z)}(\Omega)$ such that 
$$\|u_n\|_*>\left(1-\frac 1n\right)\|u_n\|\ \mbox{for all}\ n\in\NN.$$
We may assume that $\|u_n\|=1$ for all $n\in\NN$ and so we can say (at least for a subsequence) that
$$u_n\overset{w}{\to} u\ \mbox{in}\ W_0^{1,p(z)}(\Omega),\quad u_n\rightarrow u\ \mbox{in}\ L^{p(z)}(\Omega).$$
We have
$$\begin{array}{ll}
& \|u\|_*\geq 1\geq \|u\|,\\
\Rightarrow & \displaystyle \hat{\lambda}_1\rho_p(u)\geq \int_\Omega \vartheta (z)|u|^{p(z)}dz\geq \rho_p(Du),\\
\Rightarrow & u=\hat{u}_1\in {\rm int}\, C_+.
\end{array}$$
But then $\hat{\lambda}_1\rho_p(u) > \rho_p(Du)$, contradicting \eqref{eq3}.
\end{remark}

On account of hypothesis $H_0$, the function space on which we will conduct the analysis of problem \eqref{eq1} is
$W_0^{1,p(z)}(\Omega)$. Since $q_+<p_-$, we have $W_0^{1,p(z)}(\Omega) \hookrightarrow  W_0^{1,q(z)}(\Omega)$. 

By a solution of problem \eqref{eq1}, we understand a weak solution, namely a function $u\in W_0^{1,p(z)}(\Omega)$ such that
$$\langle A_{p(z)}(u),h\rangle +\langle A_{q(z)}(u),h\rangle =\int_\Omega \hat{r}(z)|Du|^{\tau(z)-1}hdz+\int_\Omega f(z,u)hdz\ \mbox{for all}\ h\in W_0^{1,p(z)}(\Omega).$$

Eventually, the anisotropic regularity theory will imply that the positive solution of problem \eqref{eq1} is in ${\rm int}\,C_+$.

\section{Analysis of the ``frozen" problem}

As we already explained in the Introduction, we will fix (``freeze") the gradient term in the reaction. So, let $v\in C_0^1(\overline{\Omega})$ and consider the Carath\'eodory function $g_v(z,x)$ defined by
$$
g_v(z,x)=\hat{r}(z)|Dv(z)|^{\tau(z)-1}+f(z,x).
$$

Using this function as forcing (source) term, we consider the following anisotropic Dirichlet problem
\begin{equation}\label{eq13}
  -\Delta_{p(z)} u(z)-\Delta_{q(z)} u(z)=g_v(z,u(z)) \mbox{ in } \Omega,\;u|_{\partial\Omega}=0,\;u\geq0.
\end{equation}

This problem is variational. Setting $G_v(z,x)=\displaystyle{\int_0^x g_v(z,s)ds}$, we consider the $C^1$-functional $\Psi_v:W^{1,p(z)}_0(\Omega)\to\RR$ defined by
$$
\Psi_v(u)=\int_\Omega \frac{1}{p(z)}|Du|^{p(z)}dz+\int_\Omega\frac{1}{q(z)}|Du|^{q(z)}dz-
\int_\Omega G_v(z,u^+)dz
$$
for all $u\in W^{1,p(z)}_0(\Omega)$.

\begin{prop}\label{prop5}
  If hypotheses $H_0$, $H_1$ hold, then the functional $\Psi_v(\cdot)$ is coercive.
\end{prop}

\begin{proof}
  We proceed indirectly. So, suppose that $\Psi_v(\cdot)$ is not coercive. Then we can find $\{u_n\}_{n\in\NN}\subseteq W^{1,p(z)}_0(\Omega)$ such that
\begin{equation}\label{eq14}
  \|u_n\|\to\infty \mbox{ as } n\to\infty \mbox{ and } \Psi_v(u_n)\leq C_2 \mbox{ for some }C_2>0, \mbox{ all }n\in\NN.
\end{equation}

From \eqref{eq13} we have
\begin{eqnarray}\nonumber
  && \int_\Omega\frac{1}{p(z)}|Du_n|^{p(z)}dz-\int_\Omega F(z,u_n^+)dz\leq C_3 \mbox{ for some }C_3>0, \mbox{ all }n\in\NN, \\ \nonumber
   &\Rightarrow& \frac{1}{p_+}\int_\Omega |Du_n|^{p(z)}dz-\int_\Omega F(z,u_n^+)dz\leq C_3 \mbox{ for all }n\in\NN, \\ \nonumber
  &\Rightarrow& \frac{1}{p_+}\int_\Omega \hat{\lambda}_1(u_n^+)^{p(z)}dz-\int_\Omega F(z,u_n^+)dz+\frac{1}{p_+}\rho_p(Du_n^-)\leq C_3 \mbox{ for all }n\in\NN\ \mbox{(see \eqref{eq3}),} \\ \nonumber
  &\Rightarrow& \frac{1}{p_+}\int_\Omega \left[\hat{\lambda}_1(u_n^+)^{p(z)}-p_+F(z,u_n^+)\right]dz+ \frac{1}{p_+}\rho_p(Du_n^-)\leq C_3 \mbox{ for all }n\in\NN, \\ \nonumber
  &\Rightarrow& \rho_p(Du_n^-)\leq C_4 \mbox{ for some }C_4>0, \mbox{ all }n\in\NN \mbox{ (see hypotesis $H_1(iii)$), } \\
  &\Rightarrow& \{u_n^-\}_{n\in\NN}\subseteq W^{1,p(z)}_0(\Omega) \mbox{ is bounded (see Proposition \ref{prop1}). } \label{eq15}
\end{eqnarray}

From \eqref{eq14} and \eqref{eq15} it follows that
\begin{eqnarray}\nonumber
  && \|u_n^+\|\to\infty, \\
  &\Rightarrow& \rho_p(Du_n^+)\to\infty \mbox{ (see Proposition \ref{prop1}).} \label{eq16}
\end{eqnarray}

We set $y_n=\frac{u_n^+}{\rho_p(Du_n^+)^{1/p(z)}}\in W^{1,p(z)}_0(\Omega)$. From proof of Lemma \ref{lem4}, we have that
\begin{eqnarray}
  && \left\{\rho_p(Dy_n)\right\}_{n\in\NN} \mbox{ is bounded, } \label{eq17} \\ \nonumber
  &\Rightarrow& \{y_n\}_{n\in \NN}\subseteq W^{1,p(z)}_0(\Omega) \mbox{ is bounded, $y_n\geq 0$ for all }n\in\NN.
\end{eqnarray}

We may assume that
\begin{equation}\label{eq18}
  y_n\overset{w}{\to}y \mbox{ in }W^{1,p(z)}_0(\Omega) \mbox{ and } y_n\to y \mbox{ in } L^{p(z)}(\Omega), \ y\geq0.
\end{equation}

From \eqref{eq14} and \eqref{eq15}, we have
\begin{eqnarray} \nonumber
  && \frac{1}{p_+}\left[\rho_p(Du_n^+)-\int_\Omega p_+ F(z,u_n^+)dz\right]\leq C_5 \mbox{ for some } C_5>0, \mbox{ all }n\in \NN, \\ \nonumber
  &\Rightarrow& \frac{1}{p_+}\left[\rho_p(Dy_n)-\hat{C}_n-\int_\Omega \frac{p_+F(z,u_n^+)}{\rho_p(Du_n^+)}dz\right]\leq \frac{C_5}{\rho_p(Du_n^+)} \mbox{ for all }n\in\NN \\ \nonumber
   &\Rightarrow& \rho_p(Dy_n)\leq \frac{p_+C_5}{\rho_p(Du_n^+)}+\int_\Omega \frac{p_+F(z,u_n^+)}{\rho_p(Du_n^+)}dz +\hat{C}_n \mbox{ for all } n\in \NN. \label{eq19}
\end{eqnarray}

From the sequential weak lower semicontinuity of the modular function $\rho_p(\cdot)$ (it is continuous and convex), we have
\begin{equation}\label{eq20}
  \rho_p(Dy)\leq \liminf_{n\to\infty} \rho_p(Dy_n) \mbox{ (see \eqref{eq17}). }
\end{equation}

Also, from \eqref{eq16} we have
\begin{equation}\label{eq21}
  \frac{C_5}{\rho_p(Du_n^+)}\to 0 \mbox{ as } n\to \infty.
\end{equation}

Moreover, note that
\begin{eqnarray*}
  \frac{p_+F(z,u_n^+)}{\rho_p(Du_n^+)} &=& \frac{p_+F(z,u_n^+)}{(u_n^+)^{p(z)}}
\left[\frac{u_n^+}{\rho_p(Du_n^+)^{1/p(z)}}\right]^{p(z)} \\
   &=& \frac{p_+F(z,u_n^+)}{(u_n^+)^{p(z)}}y_n^{p(z)} \mbox{ for all }n\in \NN.
\end{eqnarray*}

From hypothesis $H_1(ii)$, we have
\begin{eqnarray}\nonumber
  && \limsup_{n\to\infty}\frac{p_+F(z,u_n^+(z))}{(u_n^+)^{p(z)}}\leq \vartheta(z)  \mbox{ for a.a. }z\in \{y>0\},\\
  &\Rightarrow& \limsup_{n\to\infty}\int_\Omega \frac{p_+F(z,u_n^+)}{(u_n^+)^{p(z)}}y_n^{p(z)}dz\leq \int_\Omega \vartheta(z) y^{p(z)}dz \label{eq22} \\ \nonumber
 && \mbox{ (by Fatou's lemma).}
\end{eqnarray}

We return to \eqref{eq19}, pass to the limit as $n\to\infty$ and use \eqref{eq20}, \eqref{eq21} and \eqref{eq22}. We obtain
\begin{eqnarray*}
  && \rho_p(Dy)-\int_\Omega \vartheta(z)y^{p(z)}dz\leq0, \\
  &\Rightarrow& C_1\rho_p(Dy)\leq 0 \mbox{ (see Lemma 4),} \\
  &\Rightarrow& y=0 \mbox{ (by Poincar\'e's inequality).}
\end{eqnarray*}

Then from \eqref{eq17} and \eqref{eq18}, we obtain
\begin{equation}\label{eq23}
  \rho_p(Dy_n)\to0.
\end{equation}

But using \eqref{eq16} we see that
\begin{equation}\label{eq24}
  \rho_p(Dy_n)\geq C_6>0 \mbox{ for all }n\geq n_0.
\end{equation}

Comparing \eqref{eq23} and \eqref{eq24} we have a contradiction.
\end{proof}

Let $S_v^+$ denote the set of positive solutions of problem \eqref{eq13} (the ``frozen problem").

\begin{prop}\label{prop6}
  If hypotheses $H_0$, $H_1$ hold, then $\emptyset\not=S_v^+\subseteq{\rm int}\,C_+$.
\end{prop}

\begin{proof}
  From Proposition \ref{prop5}, we know that $\Psi_v(\cdot)$ is coercive. Also it is sequentially weakly lower semicontinuous. Therefore by the Weierstrass-Tonelli theorem, we can find $u_0\in W^{1,p(z)}_0(\Omega)$ such that
\begin{equation}\label{eq25}
  \Psi_v(u_0)=\min\left\{\Psi_v(u):\:u\in W^{1,p(z)}_0(\Omega)\right\}.
\end{equation}

Let $u\in{\rm int}\,C_+$ and let $\delta>0$ be as postulated by hypothesis $H_1(iv)$. We choose $t\in (0,1)$ small such that
$$
0\leq tu(z)\leq \delta \mbox{ for all }z\in \overline{\Omega}.
$$

We have
\begin{eqnarray*}
  \Psi_t(tu) &\leq& \frac{t^{q_-}}{q_-}\left[\rho_p(Du)+\rho_q(Du)\right]-
\frac{t^{\mu_+}}{\mu_+}\rho_\mu(u) \\
  && \mbox{ (see hypothesis $H_1(iv)$ and recall that $t\in(0,1)$).}
\end{eqnarray*}

Since $t\in(0,1)$ and $\mu_+<q_-<p_-$, by choosing $t\in (0,1)$ even smaller, we have
\begin{eqnarray*}
  && \Psi_v(tu)<0, \\
  &\Rightarrow& \Psi_v(u_0)<0=\Psi_v(0) \mbox{ (see \eqref{eq25}), } \\
  &\Rightarrow& u_0\not=0.
\end{eqnarray*}

From \eqref{eq25} we have
$$
\Psi'_v(u_0)=0,
$$
\begin{equation}\label{eq26}
  \Rightarrow\langle A_{p(z)}(u_0),h\rangle+\langle A_{q(z)}(u_0),h\rangle=\int_\Omega g_v(z,u_0^+)hdz \mbox{ for all }h\in W^{1,p(z)}_0(\Omega).
\end{equation}

In \eqref{eq26} we choose $h=-u_0^-\in W^{1,p(z)}_0(\Omega)$. We obtain
\begin{eqnarray*}
  && \rho_p(Du_0^-)+\rho_q(Du_0^-)\leq0, \\
  &\Rightarrow& u_0\geq0,\; u_0\not=0.
\end{eqnarray*}

From \eqref{eq26} it follows that $u_0\in S_v^+\not=\emptyset$.

Then Theorem 4.1 of Fan and Zhao \cite{10Fan-Zha} (see also Gasi\'nski and Papageorgiou \cite[Proposition 3.1]{15Gas-Pap}  and Papageorgiou, R\u adulescu and Zhang \cite[Proposition A1]{23Pap-Rad-Zha}), we have that $u_0\in L^\infty(\Omega)$. Then Corollary 3.1 of Tan and Fang \cite{29Tan-Fan} implies that $u_0\in C_+\setminus\{0\}$. Finally, the anisotropic maximum principle of Papageorgiou, R\u adulescu and Zhang \cite[Proposition A2]{23Pap-Rad-Zha} (see also Zhang \cite{31Zha}), implies that $u_0\in {\rm int}\,C_+$.

We conclude that $\emptyset\not=S_v^+\subseteq{\rm int}\,C_+$.
\end{proof}

Hypotheses $H_1(i),(ii)$ imply that
$$
0\leq f(z,x)\leq C_7\left[1+x^{p(z)-1}\right] \mbox{ for a.a. }z\in \Omega, \mbox{ all }x\geq0, \mbox{ some }C_7>0.
$$

This growth condition combined with hypothesis $H_1(iv)$, imply that given $r\in(p_+,p_-^*)$, we can find $C_8=C_8(r)>0$ such that
\begin{eqnarray}\nonumber
  && f(z,x)\geq C_1x^{\mu(z)-1}-C_8 x^{r-1} \mbox{ for a.a. }z\in \Omega, \mbox{ all }x\geq0, \\
  &\Rightarrow& g_v(z,x)\geq C_1 x^{\mu(z)-1}-C_8 x^{r-1} \mbox{ for a.a. }z\in\Omega, \mbox{ all }x\geq0 \label{eq27} \\ \nonumber
  && \mbox{ (see hypotheses $H_0$). }
\end{eqnarray}

This unilateral growth condition on $g_v(z,\cdot)$ leads to the following auxiliary Dirichlet problem
\begin{equation}\label{eq28}
    \left\{
\begin{array}{lll}
-\Delta_{p(z)} u(z)-\Delta_{q(z)}u(z)=C_1 u(z)^{\mu(z)-1}-C_8u(z)^{r-1} \text{ in } \Omega,\\
u|_{\partial\Omega}=0,\;u\geq0.
\end{array}
\right.
\end{equation}

\begin{prop}\label{prop7}
  If hypotheses $H_0$ hold, then problem \eqref{eq28} admits a unique positive solution
$$
\overline{u}\in {\rm int}\,C_+.
$$
\end{prop}

\begin{proof}
  First we show the existence of a positive solution. To this end we consider the $C^1$-functional $\tau_0:W^{1,p(z)}_0(\Omega)\to \RR$ defined by
$$
\tau_0(u)=\int_\Omega \frac{1}{p(z)}|Du|^{p(z)}dz+\int_\Omega \frac{1}{q(z)}|Du|^{q(z)}dz+\frac{C_8}{r}\|u^+\|_r^r-\int_\Omega \frac{C_1}{\mu(z)}(u^+)^{\mu(z)}dz
$$
for all $u\in W^{1,p(z)}_0(\Omega)$.

Since $\mu_+<q_-$, we see that $\tau_0(\cdot)$ is coercive. Also it is sequentially weakly lower semicontinuous. So, we can find $\overline{u}\in W^{1,p(z)}_0(\Omega)$ such that
\begin{eqnarray}
  &&\tau_0(\overline{u})=\min\left\{\tau(u):\:u\in W^{1,p(z)}_0(\Omega)\right\}<0=\tau_0(0) \label{eq29}\\ \nonumber
  &&\mbox{ (recall that $\mu_+<q_-<p_+<r$), }\\ \nonumber
  &\Rightarrow& \overline{u}\not=0.
\end{eqnarray}

From \eqref{eq29} we have
$$
\tau_0'(\overline{u})=0,
$$
\begin{equation}\label{eq30}
  \Rightarrow\langle A_{p(z)}(\overline{u}),h\rangle+\langle A_{q(z)}(\overline{u}),h\rangle=\int_\Omega C_1(\overline{u}^+)^{\mu(z)-1}hdz-\int_\Omega C_8(\overline{u}^+)^{r_+}hdz
\end{equation}
for all $h\in W^{1,p(z)}_0(\Omega)$.

In \eqref{eq30} we use $h=-\overline{u}^-\in W^{1,p(z)}_0(\Omega)$ and obtain
\begin{eqnarray*}
   && \rho_p(D\overline{u}^-)+\rho_q(D\overline{u}^-)=0, \\
  &\Rightarrow& \overline{u}\geq0,\ \overline{u}\not=0.
\end{eqnarray*}

It follows that $\overline{u}\in W^{1,p(z)}_0(\Omega)$ is a positive solution of \eqref{eq28}. As before (see the proof of Proposition \ref{prop6}), the anisotropic regularity theory and the maximum principle imply that $\overline{u}\in{\rm int}\,C_+$.

Next, we show the uniqueness of this positive solution of problem \eqref{eq28}. To this end, we introduce the integral functional $j:L^1(\Omega)\to \overline{\RR}=\RR\cup \{+\infty\}$ defined by
$$
j(u)=\left\{
       \begin{array}{ll}
         \int_\Omega \frac{1}{p(z)}\left|Du^{1/\mu_+}\right|dz +\int_\Omega \frac{1}{q(z)} \left|Du^{1/\mu_+}\right|dz, & \hbox{ if } u\geq0,\;u^{1/\mu_+}\in W^{1,p(z)}_0(\Omega), \\
         +\infty, & \hbox{ otherwise.}
       \end{array}
     \right.
$$

We set ${\rm dom}\,j=\left\{u\in L^1(\Omega):\:j(u)<+\infty\right\}$ (the effective domain of $j(\cdot)$). Theorem 2.2 of Taka\v c and Giacomoni \cite{28Tak-Gia}, implies that $j(\cdot)$ is convex.

Suppose that $\overline{v}\in W^{1,p(z)}_0(\Omega)$ is another positive solution of problem \eqref{eq28}. Again we have $\overline{v}\in{\rm int}\,C_+$. Then using Proposition 4.1.22 of Papageorgiou, R\u adulescu and Repov\v s \cite[p. 274]{21Pap-Rad-Rep}, we have
$$
\frac{\overline{u}}{\overline{v}},\ \frac{\overline{v}}{\overline{u}}\in L^\infty(\Omega).
$$

Let $h= \overline{u}^{\mu_+}-\overline{v}^{\mu_+}\in W^{1,p(z)}_0(\Omega)$. For $|t|<1$ small, we have
$$
\overline{u}^{\mu_+}+th \in {\rm dom}\,j,\ \overline{v}^{\mu_+}+th\in {\rm dom}\,j.
$$

Then the convexity of $j(\cdot)$ implies the G\^ateaux differentiability of $j(\cdot)$ at $\overline{u}^{\mu_+}$ and at $\overline{v}^{\mu_+}$ in the direction $h$. Moreover, using Green's identity we obtain
\begin{eqnarray*}
  j'(\overline{u}^{\mu_+})(h) &=& \frac{1}{\mu_+}\int_\Omega \frac{-\Delta_{p(z)}\overline{u}-\Delta_{q(z)}\overline{u}}{\overline{u}^{\mu_+-1}}hdz \\
   &=& \frac{1}{\mu_+}\int_\Omega\left[\frac{C_1}{\overline{u}^{\mu_+-\mu(z)}}
-C_8 \overline{u}^{r-\mu_+}\right]hdz \\
  j'(\overline{v}^{\mu_+})(h) &=& \frac{1}{\mu_+}\int_\Omega \frac{-\Delta_{p(z)}\overline{v}-\Delta_{q(z)}\overline{v}}{
\overline{v}^{\mu_+-1}}hdz \\
   &=& \frac{1}{\mu_+}\int_\Omega \left[\frac{C_1}{\overline{v}^{\mu_+-\mu(z)}}-C_8\overline{v}^{r-\mu_+}
\right]hdz.
\end{eqnarray*}

The convexity of $j(\cdot)$ implies the monotonicity of $j'(\cdot)$. So, we have
\begin{eqnarray*}
   && 0\leq \int_\Omega \left[\frac{1}{\overline{u}^{\mu_+-\mu(z)}}-
\frac{1}{\overline{v}^{\mu_+-\mu(z)}}\right](\overline{u}^{\mu_+}-
\overline{v}^{\mu_+})dz
 \\
   &&+ \int_\Omega C_8\left[\overline{v}^{r-\mu_+}-\overline{u}^{r-\mu_+}\right]
(\overline{u}^{\mu_+}-\overline{v}^{\mu_+})dz\leq 0, \\
  &\Rightarrow& \overline{u}=\overline{v}.
\end{eqnarray*}

Therefore $\overline{u}\in {\rm int}\,C_+$ is the unique positive solution of problem \eqref{eq28}.
\end{proof}

\begin{prop}\label{prop8}
  If hypotheses $H_0$, $H_1$ hold, then $\overline{u}\leq u$ for all $u\in S_v^+$.
\end{prop}

\begin{proof}
  Let $u\in S_v^+$. We introduce the Carath\'eodory function $k(z,x)$ defined by
\begin{equation}\label{eq31}
  k(z,x)=\left\{
           \begin{array}{ll}
             C_1(x^+)^{\mu(z)-1}-C_8(x^+)^{r-1}, & \hbox{ if } x\leq u(z) \\
             C_1u(z)^{\mu(z)-1}-C_8u(z)^{r-1}, & \hbox{ if } u(z)<x.
           \end{array}
         \right.
\end{equation}

We set $K(z,x)=\displaystyle{\int_0^x k(z,s)ds}$ and consider the $C^1$-functional $\hat{\tau}: W^{1,p(z)}_0(\Omega)\to\RR$ defined by
$$
\hat{\tau}(u)=\int_\Omega \frac{1}{p(z)}|Du|^{p(z)}dz+\int_\Omega \frac{1}{q(z)}|Du|^{q(z)}dz-\int_\Omega K(z,u)dz
$$
for all $u\in W^{1,p(z)}_0(\Omega)$.

From \eqref{eq31} and Poincar\'e's inequality, we see that $\hat{\tau}(\cdot)$ is coercive. Also, it is sequentially weakly lower semicontinuous. So, we can find $\hat{u}\in W^{1,p(z)}_0(\Omega)$ such that
\begin{eqnarray}
  && \hat{\tau}(\hat{u})=\min\left\{\hat{\tau}(u):\:u\in W^{1,p(z)}_0(\Omega)\right\}<0=\hat{\tau}(0) \label{eq32}\\ \nonumber
  && \mbox{ (as before, since $\mu_+<q_-<p_+<r$), } \\ \nonumber
  &\Rightarrow& \hat{u}\not=0.
\end{eqnarray}

From \eqref{eq32} we have
$$
\hat{\tau}'(\hat{u})=0,
$$
\begin{equation}\label{eq33}
  \Rightarrow \langle A_{p(z)}(\hat{u}),h\rangle+\langle A_{q(z)}(\hat{u}),h\rangle=\int_\Omega k(z,\hat{u})hdz \mbox{ for all }h\in W^{1,p(z)}_0(\Omega).
\end{equation}

In \eqref{eq33} first we choose $h=-\hat{u}^-\in W^{1,p(z)}_0(\Omega)$ and obtain $\hat{u}\geq0$, $\hat{u}\not=0$. Next, in \eqref{eq33} we use $h=(\hat{u}-u)^+\in W^{1,p(z)}_0(\Omega)$. We have
\begin{eqnarray*}
  && \langle A_{p(z)}(\hat{u}),(\hat{u}-u)^+\rangle+\langle A_{q(z)}(\hat{u}), (\hat{u}-u)^+\rangle \\
  &=& \int_\Omega\left[C_1 u^{\mu(z)-1}-C_8u^{r-1}\right](\hat{u}-u)^+dz \mbox{ (see \eqref{eq31}) } \\
  &\leq& \int_\Omega g_v(z,u)(\hat{u}-u)^+dz \mbox{ (see \eqref{eq27}) } \\
  &=& \langle A_{p(z)}(u), (\hat{u}-u)^+\rangle+\langle A_{q(z)}, (\hat{u}-u)^+\rangle \mbox{ (since $u\in S_v^+$), } \\
  \Rightarrow&& \hat{u}\leq u.
\end{eqnarray*}

So, we have proved that
\begin{equation}\label{eq34}
  \hat{u}\in[0,u],\ \hat{u}\not=0.
\end{equation}

From \eqref{eq34}, \eqref{eq31}, \eqref{eq33} and Proposition \ref{prop7} we infer that
\begin{eqnarray*}
  && \hat{u}=\overline{u} \in {\rm int}\,C_+, \\
  &\Rightarrow& \overline{u}\leq u \mbox{ for all }u\in S_v^+ \mbox{ (see \eqref{eq34}). }
\end{eqnarray*}
The proof is now complete.
\end{proof}

Using this lower bound, we can show that $S_v^+$ has a smallest element (minimal positive solution). So, we have a canonical way to choose an element from the solution set $S_v^+$ as $v$ varies (a selection of the solution multifunction $v\mapsto S_v^+$).

\begin{prop}\label{prop9}
  If hypotheses $H_0$, $H_1$ hold, then there exists $\tilde{u}_v\in S_v^+$ such that $\tilde{u}_v\leq u$ for all $u\in S_v^+$.
\end{prop}

\begin{proof}
  From Papageorgiou, R\u adulescu and Repov\v s \cite{20Pap-Rad-Rep} (see the proof of Proposition \ref{prop7}), we know that $S_v^+$ is downward directed. So using Lemma 3.10 of Hu and Papageorgiou \cite[p. 178]{17Hu-Pap}, we can find a decreasing sequence $\{u_n\}_{n\in \NN}\subseteq S_v^+$ such that
$$
\inf_{n\in\NN} u_n=\inf S_v^+.
$$

We have
\begin{equation}\label{eq35}
  \langle A_{p(z)}(u_n),h\rangle+\langle A_{q(z)}(u_n),h\rangle=\int_\Omega g_v(z,u_n)hdz \mbox{ for all }h\in W^{1,p(z)}_0(\Omega), \mbox{ all }n\in \NN,
\end{equation}
\begin{equation}\label{eq36}
  \overline{u}\leq u_n\leq u_1 \mbox{ for all } n\in \NN \mbox{ (see Proposition \ref{prop8}). }
\end{equation}

In \eqref{eq35} we choose $h=u_n \in W^{1,p(z)}_0(\Omega)$. Using \eqref{eq36} and hypothesis $H_1(i)$ we infer that
$$
\{u_n\}_{n\in \NN}\subseteq W^{1,p(z)}_0(\Omega) \mbox{ is bounded. }
$$

So, by passing to a suitable subsequence if necessary, we may assume that
\begin{equation}\label{eq37}
  u_n\overset{w}{\to}\tilde{u}_v \mbox{ in } W^{1,p(z)}_0(\Omega) \mbox{ and } u_n\to \tilde{u}_v \mbox{ in } L^{p(z)}(\Omega).
\end{equation}

In \eqref{eq35} we choose the test function $h=u_n-\tilde{u}_n\in W^{1,p(z)}_0(\Omega)$, pass to the limit as $n\to\infty$ and use \eqref{eq37}. Then
\begin{eqnarray}\nonumber
   && \lim_{n\to\infty}\left[\langle A_{p(z)}(u_n),u_n-\tilde{u}_v\rangle+\langle A_{q(z)}(u_n), u_n-\tilde{u}_v\rangle\right]=0, \\ \nonumber
  &\Rightarrow& \limsup_{n\to\infty}\left[\langle A_{p(z)}(u_n), u_n-\tilde{u}_v\rangle+\langle A_{q(z)}(u), u_n-\tilde{u}_v\right]\leq 0 \\ \nonumber
  && \mbox{ (since $A_{q(z)}(\cdot)$ is monotone), } \\ \nonumber
  &\Rightarrow& \limsup_{n\to\infty}\langle A_{p(z)}(u_n),u_n-\tilde{u}_v\rangle\leq 0 \mbox{ (see \eqref{eq37}), } \\
  &\Rightarrow& u_n\to \tilde{u}_v \mbox{ in } W^{1,p(z)}_0(\Omega) \mbox{ (see Proposition \ref{prop2}). } \label{eq38}
\end{eqnarray}

In \eqref{eq35} we pass to the limit as $n\to\infty$ and use \eqref{eq38}. Then
\begin{eqnarray*}
  && \langle A_{p(z)}(\tilde{u}_v),h\rangle+\langle A_{q(z)}(\tilde{u}_v),h\rangle= \int_\Omega g_v(z,\tilde{u}_v)hdz \mbox{ for all }h\in   W^{1,p(z)}_0(\Omega), \\
  && \overline{u}\leq \tilde{u}_v.
\end{eqnarray*}

In follows that $\tilde{u}_v\in S_v^+$ and $\tilde{u}_v=\inf S_v^+$.
\end{proof}

So, we define the minimal solution map $\beta:C_0^1(\overline{\Omega})\to C_0^1(\overline{\Omega})$ by
$$
\beta(v)=\tilde{u}_v\in{\rm int}\,C_+ \mbox{ for all }v\in C_0^1(\overline{\Omega}).
$$

Clearly, a fixed point of this map will be the positive solution of problem \eqref{eq1}. To produce a fixed point of $\beta(\cdot)$, we  use the Leray-Schauder Alternative Principle (see Theorem \ref{th3}). This theorem requires that the minimal solution map $\beta(\cdot)$ is compact. We prove this property in the next section.

\section{The minimal solution map}

In this section we show that the minimal solution map $\beta:C_0^1(\overline{\Omega})\to C_0^1(\overline{\Omega})$ is compact. To this end the following proposition is helpful.

\begin{prop}\label{prop10}
  If hypotheses $H_0$, $H_1$ hold, $v_n\to v$ in $C_0^1(\overline{\Omega})$ and $u\in S_v^+$, then we can find $u_n\in S_{v_n}^+$ $n\in \NN$ such that $u_n\to u$ in $C_0^1(\overline{\Omega})$.
\end{prop}

\begin{proof}
First we consider the following anisotropic Dirichlet problem
\begin{equation}\label{eq39}
  -\Delta_{p(z)} y(z)-\Delta_{q(z)}y(z)=g_{v_n}(z,u(z)) \mbox{ in }\Omega,\ y|_{\partial\Omega}=0,\ y\geq0,\ n\in \NN.
\end{equation}

Hypotheses $H_0$, $H_1(i)$ imply that $g_{v_n}(\cdot,u(\cdot))\in L^\infty(\Omega)$. We consider the nonlinear operator $V: W^{1,p(z)}_0(\Omega)\to  W^{-1,p'(z)}(\Omega)= W^{1,p(z)}_0(\Omega)^*$ defined by
$$
V(u)=A_{p(z)}(u)+A_{q(z)}(u) \mbox{ for all }u\in  W^{1,p(z)}_0(\Omega).
$$

By Proposition \ref{prop2}, $V(\cdot)$ is continuous, strictly monotone (thus, maximal monotone too) and also we have
\begin{eqnarray*}
  && \langle V(u),u\rangle=\rho_p(Du)+\rho_q(Du) \mbox{ for all }u\in  W^{1,p(z)}_0(\Omega), \\
  &\Rightarrow& V(\cdot) \mbox{ is coercive (see Proposition \ref{prop1}).}
\end{eqnarray*}

Therefore $V(\cdot)$ is surjective (see Corollary 2.8.7 of Papageorgiou, R\u adulescu and Repov\v s \cite[p. 135]{21Pap-Rad-Rep}). So, we can find $y_n\in  W^{1,p(z)}_0(\Omega)$, $y_n\geq0$, $y_n\not=0$ such that
\begin{equation}\label{eq40}
  V(y_n)=N_{g_{v_n}}(u) \mbox{ for all } n\in \NN.
\end{equation}

The strict monotonicity of $V(\cdot)$ implies that this solution of problem \eqref{eq39} is unique. On \eqref{eq40} we act with $y_n\in  W^{1,p(z)}_0(\Omega)$ and obtain that $\{y_n\}_{n\in\NN}\subseteq  W^{1,p(z)}_0(\Omega)$ is bounded. Then the anisotropic regularity theory (see \cite{10Fan-Zha}, \cite{23Pap-Rad-Zha}) implies that
\begin{equation}\label{eq41}
  y_n\in L^\infty(\Omega),\; \|y_n\|_\infty\leq C_9 \mbox{ for some } C_9>0, \mbox{ all }n\in\NN.
\end{equation}

Then invoking Lemma 3.3 of Fukagai and Narukawa \cite{12Fuk-Nar}, we can find $\alpha\in(0,1)$ and $C_{10}>0$ such that
$$
y_n\in C_0^{1,\alpha}(\overline{\Omega}),\;\|y_n\|_{C_0^{1,\alpha}(\overline{\Omega})}
\leq C_{10} \mbox{ for all }n\in \NN.
$$

We know that $C_0^{1,\alpha}(\overline{\Omega})\hookrightarrow C_0^1(\overline{\Omega})$ compactly. Therefore by passing to a subsequence if necessary, we can have
\begin{equation}\label{eq42}
  y_n\to\tilde{u} \mbox{ in } C_0^1(\overline{\Omega}) \mbox{ as }n\to\infty.
\end{equation}

Passing to the limit as $n\to\infty$ in \eqref{eq40} and using \eqref{eq42}, we obtain
\begin{eqnarray*}
  && V(\tilde{u})=N_{g_v}(u). \\
  &\Rightarrow& \tilde{u}=u \mbox{ (from the uniqueness of the solution). }
\end{eqnarray*}

Therefore for the original sequence we have
$$
y_n\to u \mbox{ in } C_0^1(\overline{\Omega}).
$$

Next, we consider the following anisotropic Dirichlet problem
$$
-\Delta_{p(z)}w(z)-\Delta_{q(z)}w(z)=g_{v_n}(z,y_n(z)) \mbox{ in }\Omega,\; w|_{\partial\Omega}=0,\;w\geq0.
$$

Reasoning as above, we infer that this problem has a unique positive solution $w_n^1\in {\rm int}\,C_+$ and
$$
w_n^1\to u \mbox{ in } C_0^1(\overline{\Omega}) \mbox{ as } n\to\infty.
$$

Setting $w_n^0=y_n$ and continuing this way, we generate a sequence $\{w_n^k\}_{n\in \NN_0}\subseteq{\rm int}\,C_+$ such that
\begin{eqnarray}
  && V(w_n^k)=N_{g_{v_n}}(w_n^{k-1}) \mbox{ in }  W^{1,p(z)}_0(\Omega)^* \mbox{ for all }k,n\in \NN, \label{eq43} \\
  && w_n^k\to u \mbox{ in } C_0^1(\overline{\Omega}) \mbox{ as }n\to\infty \mbox{ for every }k\in \NN. \label{eq44}
\end{eqnarray}

\smallskip
{\it
Claim.} For every $n\in \NN$, the sequence $\{w_n^k\}_{k\in\NN}\subseteq  W^{1,p(z)}_0(\Omega)$ is bounded.

To prove the Claim, we argue by contradiction. So, suppose that at least for a subsequence, we have
$$
\|w_n^k\|\to\infty \mbox{ as }k\to+\infty.
$$

Then, we can say that
\begin{equation}\label{eq45}
  \rho_p(Dw_n^k)\to+\infty \mbox{ as }k\to+\infty,\: \left\{\rho_p(Dw_n^k)\right\}_{k\in \NN} \mbox{ is nondecreasing. }
\end{equation}

We set $\hat{x}_k=\frac{w_n^k}{\rho_p(Dw_n^k)^{1/p(z)}}\in  W^{1,p(z)}_0(\Omega)$, $k\in \NN_0$. On \eqref{eq43} we act with $w_n^k\in  W^{1,p(z)}_0(\Omega)$ and obtain

\begin{eqnarray}\nonumber
  && \rho_p(Dw_n^k)+\rho_q(Dw_n^k)=\int_\Omega g_{v_n}(z,w_n^{k-1}) w_n^k dz \\ \nonumber
  &\Rightarrow& \rho_p(D\hat{x}_k)-\hat{C}_k+\int_\Omega \frac{1}{\rho_p(Dw_n^k)^{\frac{p(z)-q(z)}{p(z)}}}\left[
\frac{|Dw_n^k|}{\rho_p(Dw_n^k)^{\frac{1}{p(z)}}}\right]^{q(z)}dz \\ \nonumber
  &=& \int_\Omega \frac{g_{v_n}(z,w_n^{k-1})}{\rho_p(Dw_n^k)^{1-\frac{1}{p(z)}}}\hat{x}_k dz \\ \nonumber
   &\leq& \int_\Omega \frac{g_{v_n}(z,w_n^{k-1})}{\rho_p(Dw_n^{k-1})^{1-\frac{1}{p(z)}}}\hat{x}_k dz \mbox{ (see \eqref{eq45}) } \\ \nonumber
   &=& \int_\Omega \frac{g_{v_n}(z,w_n^{k-1})}{(w_n^{k-1})^{p(z)-1}}\left[
\frac{w_n^{k-1}}{\rho_p(Dw_n^{k-1})^{\frac{1}{p(z)}}}\right]^{p(z)-1}
\hat{x}_kdz \\ \nonumber
  &=& \int_\Omega \frac{g_{v_n}(z,w_n^{k-1})}{(w_n^{k-1})^{p(z)-1}}\hat{x}_{k-1}^{p(z)-1} \hat{x}_k dz,\\
  \Rightarrow&& \rho(D\hat{x_k})\leq \int_\Omega \frac{g_{v_n}(z,w_n^{k-1})}{(w_n^{k-1})^{p(z)-1}}\hat{x}_{k-1}^{p(z)-1}
\hat{x}_k dz +\hat{C}_k \mbox{ for all }k\in\NN. \label{eq46}
\end{eqnarray}

From the proof of Lemma \ref{lem4}, we know that
\begin{equation}\label{eq47}
  0<C_{11}\leq \|\hat{x}_k\|\leq C_{12} \mbox{ for all }k\geq k_0,\mbox{ some } 0<C_{11}\leq C_{12}.
\end{equation}

So we may assume that
\begin{equation}\label{eq48}
  \hat{x}_k\overset{w}{\to}\hat{x} \mbox{ in }  W^{1,p(z)}_0(\Omega) \mbox{ and } \hat{x}_k\to \hat{x} \mbox{ in }L^{p(z)}(\Omega) \mbox{ as }k\to\infty.
\end{equation}

We return to \eqref{eq46}, pass to the limit as $k\to\infty$ and use \eqref{eq48}, the fact that $\rho_p(\cdot)$ is sequentially weakly lower semicontinuous (being continuous convex) and hypothesis $H_1(ii)$. We obtain
\begin{eqnarray*}
  && \rho_p(D\hat{x})\leq \int_\Omega \vartheta(z)\hat{x}^{p(z)}dz, \\
  &\Rightarrow& C_1\rho_p(D\hat{x})\leq 0 \mbox{ (see Lemma \ref{lem4}), } \\
  &\Rightarrow& \hat{x}=0.
\end{eqnarray*}

So, from \eqref{eq46} we have that
\begin{eqnarray*}
  && \rho_p(D\hat{x}_k)\to 0 \mbox{ as }k\to\infty, \\
  &\Rightarrow& \hat{x}_k \to0 \mbox{ in }  W^{1,p(z)}_0(\Omega) \mbox{ (see Proposition \ref{prop1}), }
\end{eqnarray*}
which contradicts \eqref{eq47}.

This proves the Claim.

As before, using the Claim and the anisotropic regularity theory, we can find $\alpha\in(0,1)$ and $C_{13}>0$ such that
$$
w_n^k\in C_0^{1,\alpha}(\overline{\Omega}),\;\|w_n^k\|_{ C_0^{1,\alpha}(\overline{\Omega})}\leq C_{13} \mbox{ for all }k\in \NN.
$$

Since $ C_0^{1,\alpha}(\overline{\Omega})\hookrightarrow C_0^1(\overline{\Omega})$ compactly, at least for a subsequence we have
$$
w_n^k\to u_n \mbox{ in } C_0^1(\overline{\Omega}) \mbox{ as } k\to\infty,\;u_n\not=0 \mbox{ for all }n\in\NN.
$$

From \eqref{eq43} in the limit as $k\to \infty$, we obtain
\begin{eqnarray}\nonumber
  &&  V(u_n)=N_{g_{v_n}}(u_n) \mbox{ in } W^{1,p(z)}_0(\Omega)^* \mbox{ for all }n\in\NN,\\
  &\Rightarrow& u_n\in S_{v_n}^+ \mbox{ for all }n\in\NN. \label{eq49}
\end{eqnarray}

As we did in the proof of the Claim, via a contradiction argument, we also show that $\{u_n\}_{n\in\NN}\subseteq W^{1,p(z)}_0(\Omega)$ is bounded and from this we infer that $\{u_n\}_{n\in\NN}\subseteq C_0^1(\overline{\Omega})$ is relatively compact. So, we can say that
$$
u_n\to \tilde{u} \mbox{ in } C_0^1(\overline{\Omega}) \mbox{ as }n\to\infty.
$$

The double limit lemma (see Proposition A.2.35 of Gasi\'nski and Papageorgiou \cite[p. 906]{14Gas-Pap}), implies that $\tilde{u}=u$ and so finally we have
$$
u_n\to u \mbox{ in } C_0^1(\overline{\Omega}) \mbox{ and }u_n\in S_v^+ \mbox{ for all }n\in \NN \mbox{ (see \eqref{eq49}) }.
$$
The proof is now complete.
\end{proof}

Using this proposition, we can show the compactness of the  minimal solution map $\beta(\cdot)$.

\begin{prop}\label{prop11}
  If hypotheses $H_0$, $H_1$ hold, then the minimal map $\beta(\cdot): C_0^1(\overline{\Omega})\to C_0^1(\overline{\Omega})$ is compact.
\end{prop}

\begin{proof}
First we show that $\beta(\cdot)$ maps bounded sets in $C_0^1(\overline{\Omega})$ onto relatively compact subsets of $C_0^1(\overline{\Omega})$.

So, let $B\in C_0^1(\overline{\Omega})$ be bounded. Hypotheses $H_1(i),(ii)$ imply that given $\varepsilon>0$, we can find $C_\varepsilon>0$ such that
\begin{equation}\label{eq50}
  0\leq f(z,x)\leq [\vartheta(z)+\varepsilon]x^{p(z)-1}+C_\varepsilon \mbox{ for a.a. }z\in\Omega, \mbox{ all }x\geq0.
\end{equation}

For $v\in B$, we write $\beta(v)=\tilde{u}_v\in {\rm int}\, C_+$. We have
\begin{eqnarray}\nonumber
  &&\langle A_{p(z)}(\tilde{u}_v),h\rangle+\langle A_{q(z)}(\tilde{u}_v),h\rangle  \\ \nonumber
  &=& \int_\Omega \left[\hat{r}(z)|Dv|^{\tau(z)-1}+f(z,\tilde{u}_v)\right]hdz \\
  &\leq& \int_\Omega\left[\hat{r}(z)|Dv|^{\tau(z)-1}+(\vartheta(z)+\varepsilon)
\tilde{u}_v^{\tau(z)-1}+C_\varepsilon\right]hdz \label{eq51}\\ \nonumber
  && \mbox{ for all }h\in W^{1,p(z)}_0(\Omega),\ h\geq0 \mbox{ see \eqref{eq50}. }
\end{eqnarray}

In \eqref{eq51} we choose $h=\tilde{u}_v\in W^{1,p(z)}_0(\Omega)$, $h\geq0$. We obtain
\begin{eqnarray} \nonumber
  && \rho_p(D\tilde{u}_v)-\int_\Omega \vartheta(z)\tilde{u}_v^{p(z)}dz-\varepsilon\rho_p(\tilde{u}_v)\leq C_{14}\left[1+\rho_p(D\tilde{u}_v)^{1/p_-}\right] \\ \nonumber
  && \mbox{ for some } C_{14}>0, \\
  &\Rightarrow& \left[C_1-\frac{\varepsilon}{\hat{\lambda}_1}\right]\rho_p(D\tilde{u}_v)\leq C_{14}\left[1+\rho_p(D\tilde{u}_v)^{1/p_-}\right] \label{eq52}\\ \nonumber
  && \mbox{ (see \eqref{eq3} and Lemma \ref{lem4}). }
\end{eqnarray}

Choosing $\varepsilon\in \left(0,\hat{\lambda}_1 C_1\right)$, since $p_->1$, from \eqref{eq52} we infer that
\begin{eqnarray*}
  && \left\{\rho_p(D\tilde{u}_v)\right\}_{v\in B}\subseteq \RR_+ \mbox{ is bounded, } \\
  &\Rightarrow& \{\tilde{u}_v\}_{v\in B} \subseteq W^{1,p(z)}_0(\Omega) \mbox{ is bounded (see Proposition \ref{prop1}).}
\end{eqnarray*}

From this as before we obtain that
$$
\left\{\tilde{u}_v=\beta(v)\right\}_{v\in B} \subseteq C_0^1(\overline{\Omega}) \mbox{ is relatively compact.}
$$

Next, we show that $\beta(\cdot)$ is continuous. Suppose that $v_n\to v$ in $C_0^1(\overline{\Omega})$. According to Proposition \ref{prop10}, we can find $u_n\in S_{v_n}^+\subseteq{\rm int}\,C_+$ $n\in \NN$ such that
\begin{equation}\label{eq53}
  u_n\to \beta(v)=\tilde{u}_v \mbox{ in } C_0^1(\overline{\Omega}) \mbox{ as }n\to \infty.
\end{equation}

From the first part of the proof, we have that
$$
\left\{\beta(v_n)\right\}_{n\in\NN} \subseteq C_0^1(\overline{\Omega}) \mbox{ is relatively compact. }
$$

So, for at least a subsequence we have
\begin{equation}\label{eq54}
  \beta(v_n)\to \tilde{u}_* \mbox{ in } C_0^1(\overline{\Omega}) \mbox{ as } n\to\infty.
\end{equation}

Recall that $\overline{u}\leq \beta(v_n)$ for all $n\in\NN$ (see Proposition \ref{prop8}). Hence $\overline{u}\leq \tilde{u}_*$ and so using \eqref{eq54} we conclude that $\tilde{u}_*\in S_v^+\subseteq{\rm int}\,C_+$. Then
\begin{equation}\label{eq55}
  \beta(v)\leq \tilde{u}_* .
\end{equation}

On the other hand, we have
\begin{eqnarray*}
  && \beta(v_n)\leq u_n \mbox{ for all }n\in\NN, \\
  &\Rightarrow& \tilde{u}_*\leq \beta(v)=\tilde{u}_v \mbox{ (see \eqref{eq54} and \eqref{eq53}) } \\
  &\Rightarrow& \tilde{u}_*=\beta(v) \mbox{ (see \eqref{eq55}) } \\
  &\Rightarrow& \beta(v_n)\to \beta(v) \mbox{ in } C_0^1(\overline{\Omega}) \mbox{ (see \eqref{eq54}), }  \\
  &\Rightarrow& \beta(\cdot) \mbox{ is compact. }
\end{eqnarray*}
The proof is now complete.
\end{proof}

\section{Positive solution}

In this section using the Leray-Schauder Alternative Principle (see Theorem \ref{th3}) on the minimal solution map $\beta(\cdot)$, we produce a fixed point which is a positive solution of problem \eqref{eq1}.

We introduce the set
$$
D=\{u\in C_0^1(\overline{\Omega}):\:u=t\beta(u),\ 0<t<1\}.
$$

\begin{prop}\label{prop12}
  If hypotheses $H_0$, $H_1$ hold, then $D\subseteq C_0^1(\overline{\Omega})$ is bounded.
\end{prop}

\begin{proof}
  Let $u\in D$. Then $u\in{\rm int}\,C_+$ and
$$
\frac{1}{t}u=\beta(u) \mbox{ with }0<t<1.
$$

So, we have
\begin{eqnarray}\nonumber
  && \langle A_{p(z)}\left(\frac{1}{t}u\right),h\rangle+\langle A_{q(z)} \left(\frac{1}{t}u\right),h\rangle \\
  &=& \int_\Omega \hat{r}(z)|Du|^{\tau(z)-1}hdz+\int_\Omega f\left(z,\frac{1}{t}u\right)hdz \label{eq56}\\ \nonumber
  && \mbox{ for all }h\in W^{1,p(z)}_0(\Omega).
\end{eqnarray}

In \eqref{eq56} we choose $h=u\in  W^{1,p(z)}_0(\Omega)$. Then
\begin{eqnarray*}
  \frac{1}{t^{p_-}}\rho_p(Du)&\leq&  \int_\Omega \hat{r}(z)|Du|^{\tau(z)-1}udz+\int_\Omega f\left(z,\frac{1}{t}u\right)udz \\
   &\leq& \int_\Omega \hat{r}(z)|Du|^{\tau(z)-1}udz+\frac{1}{t^{p_+-1}}\int_\Omega f(z,u)udz \\
   && \mbox{ (see hypothesis $H_1(v)$),} \\
  \Rightarrow \rho_p(Du)&\leq&  \int_\Omega \hat{r}(z)|Du|^{\tau(z)-1}udz+t^{1-(p_+-p_-)}\int_\Omega f(z,u)udz \\
   &\leq& \int_\Omega \hat{r}(z)|Du|^{\tau(z)-1}udz+\int_\Omega f(z,u)udz \\
  && \mbox{ (since, by hypotheses $H_0$, $0\leq p_+-p_-\leq 1$ and $t\in(0,1)$), } \\
   \Rightarrow\rho_p(Du)&-&\int_\Omega \vartheta(z)u^{p(z)}dz-\varepsilon\rho_p(u)\leq \int_\Omega \hat{r}(z)|Du|^{\tau(z)-1}udz +C_{15} \\
  && \mbox{ for some $C_{15}=C_{15}(\varepsilon)>0$ (see \eqref{eq50}).}
\end{eqnarray*}

Using Lemma \ref{lem4} and choosing $\varepsilon\in(0,\hat{\lambda}_1C_1)$ (see \eqref{eq3}), we obtain
\begin{equation}\label{eq57}
  \rho_p(Du)\leq C_{16}\left[1+\int_\Omega \hat{r}(z)|Du|^{\tau(z)-1}udz\right] \mbox{ for some }C_{16}>0.
\end{equation}

Using H\"older's inequality and the Sobolev embedding theorem (see Section 2), we have
\begin{equation}\label{eq58}
  \int_\Omega \hat{r}(z)|Du|^{\tau(z)-1}udz\leq C_{17}\||Du|^{\tau(z)-1}\|_{\tau'(z)}\|u\|
\end{equation}
for some $C_{17}>0$.

We may assume that
\begin{equation}\label{eq59}
  \|\,|Du|^{\tau(z)-1}\|_{\tau'(z)}\geq1,\;\|u\|\geq1.
\end{equation}

Then from \eqref{eq58} it follows that
\begin{eqnarray}\nonumber
  && \int_\Omega \hat{r}(z)|Du|^{\tau(z)-1}udz \\ \nonumber
  &\leq& C_{18}\rho_{\tau'}\left(|Du|^{\tau(z)-1}\right)^{\frac{1}{\tau'_+}}\|u\| \mbox{ for some } C_{18}>0. \\ \nonumber
  && \mbox{ (see \eqref{eq59} and Proposition \ref{prop1})} \\
  &\leq& C_{19}\rho_{\tau'}\left(|Du|^{\tau(z)-1}\right)^{\frac{1}{\tau'_+}}
\rho_p(Du)^{\frac{1}{p_+}} \label{eq60}\\ \nonumber
  && \mbox{ (using the Poincar\'e inequality, \eqref{eq59} and Proposition \ref{prop1}). }
\end{eqnarray}

We return to \eqref{eq57} and use \eqref{eq60}. We obtain
\begin{eqnarray}\nonumber
  \rho_p(Du)^{1-\frac{1}{p_+}} &\leq& C_{20}\left[1+\|Du\|_{\tau_+}^{\tau_+-1}\right] \\ \nonumber
  && \mbox{ for some $C_{20}>0$ (see Proposition \ref{prop1}), } \\
  \Rightarrow \|u\|^{\frac{p_-}{p_+}(p_+-1)} &\leq& C_{21} \left[1+\|u\|^{\tau_+-1}\right] \label{eq61} \\ \nonumber
  && \mbox{ for some $C_{21}>0$ (see \eqref{eq59} and Proposition \ref{prop1}). }
\end{eqnarray}

But $\frac{p_-}{p_+}[p_+-1]\geq p_- -1>\tau_+-1$ (see hypotheses $H_0$). So, from \eqref{eq61} it follows that
$$
D\subseteq W^{1,p(z)}_0(\Omega) \mbox{ is bounded. }
$$

As before (see the proof of Proposition \ref{prop6}), using the anisotropic regularity theory, we infer that
$$
D\subseteq C_0^1(\overline{\Omega}) \mbox{ is relatively compact, thus bounded. }
$$
The proof is now complete.
\end{proof}

Therefore we have proved that
\begin{itemize}
  \item the minimal solution map $\beta(\cdot)$ is compact (see Proposition \ref{prop11});
  \item $D\subseteq C_0^1(\overline{\Omega})$ is bounded (see Proposition \ref{prop12}).
\end{itemize}

So, we can apply Theorem \ref{th3} (the Leray-Schauder Alternative Principle) and have the following existence theorem for problem \eqref{eq1}.

\begin{theorem}
  If hypotheses $H_0$, $H_1$ hold, then problem \eqref{eq1} has a positive solution $\hat{u}\in {\rm int}\,C_+$, $\overline{u}\leq\hat{u}$.
\end{theorem}

\medskip
\subsection*
{Acknowledgments}  The authors would like to thank a knowledgeable referee for his/her corrections and remarks, which improved the initial version of this paper.
This research was supported by the Slovenian Research Agency grants
P1-0292, J1-8131, N1-0064, N1-0083, and N1-0114.  The research of Vicen\c tiu D. R\u adulescu was supported by a grant of the Romanian Ministry of Research, Innovation and Digitization, CNCS/CCCDI--UEFISCDI, project number PCE 137/2021, within PNCDI III.

\end{document}